\documentclass[oneside,reqno,english]{amsart}
\usepackage[T1]{fontenc}
\usepackage[utf8]{inputenc}
\usepackage{xcolor}
\usepackage{babel}
\usepackage{prettyref}
\usepackage{mathrsfs}
\usepackage{mathtools}
\usepackage{amstext}
\usepackage{amsthm}
\usepackage{amssymb}
\usepackage[pdfusetitle,
 bookmarks=true,bookmarksnumbered=false,bookmarksopen=false,
 breaklinks=false,pdfborder={0 0 0},pdfborderstyle={},backref=false,colorlinks=false]
 {hyperref}
\hypersetup{
 colorlinks=true,citecolor=blue,linkcolor=blue,linktocpage=true}

\makeatletter
\numberwithin{equation}{section}
\numberwithin{figure}{section}

\newrefformat{cor}{Corollary~\ref{#1}}
\newrefformat{subsec}{Section~\ref{#1}}
\newrefformat{lem}{Lemma~\ref{#1}}
\newrefformat{thm}{Theorem~\ref{#1}}
\newrefformat{sec}{Section~\ref{#1}}
\newrefformat{chap}{Chapter~\ref{#1}}
\newrefformat{prop}{Proposition~\ref{#1}}
\newrefformat{exa}{Example~\ref{#1}}
\newrefformat{tab}{Table~\ref{#1}}
\newrefformat{rem}{Remark~\ref{#1}}
\newrefformat{def}{Definition~\ref{#1}}
\newrefformat{fig}{Figure~\ref{#1}}
\newrefformat{claim}{Claim~\ref{#1}}

\makeatother

\theoremstyle{plain}
\newtheorem{thm}{\protect\theoremname}[section]
\theoremstyle{remark}
\newtheorem{rem}[thm]{\protect\remarkname}
\theoremstyle{plain}
\newtheorem{cor}[thm]{\protect\corollaryname}
\theoremstyle{definition}
\newtheorem{defn}[thm]{\protect\definitionname}
\newtheorem{example}[thm]{\protect\examplename}
\theoremstyle{plain}
\newtheorem{lem}[thm]{\protect\lemmaname}
\newtheorem{prop}[thm]{\protect\propositionname}
\newrefformat{cor}{Corollary \ref{#1}}
\newrefformat{def}{Definition \ref{#1}}
\newrefformat{exa}{Example \ref{#1}}
\newrefformat{lem}{Lemma \ref{#1}}
\newrefformat{prop}{Proposition \ref{#1}}
\newrefformat{rem}{Remark \ref{#1}}
\newrefformat{sec}{Section~\ref{#1}}
\newrefformat{subsec}{Section~\ref{#1}}
\newrefformat{thm}{Theorem \ref{#1}}
\providecommand{\corollaryname}{Corollary}
\providecommand{\definitionname}{Definition}
\providecommand{\examplename}{Example}
\providecommand{\lemmaname}{Lemma}
\providecommand{\propositionname}{Proposition}
\providecommand{\remarkname}{Remark}
\providecommand{\theoremname}{Theorem}

\begin{document}
\title{Moment kernels, nested defects, and Cuntz dilations}
\begin{abstract}
Random operator tuples possess a rich second-moment structure that
is not visible at the level of pointwise operator inequalities. This
paper shows that their averaged word moments form a positive kernel
whose behavior is controlled by a single shift-positivity condition.
When this condition holds, the kernel admits a Cuntz dilation, and
all mean-square interactions are realized inside a canonical isometric
model. This leads to a mean-square version of the free von Neumann
inequality and to a free functional calculus for random tuples. We
further introduce a hierarchy of higher-order defects of the moment
kernel and prove that their positivity is equivalent to the existence
of a nested chain of projections inside one Cuntz dilation. This yields
a multi-level decomposition of moment structure, a Wold-type splitting
into dissipative and unitary parts, and a curvature-type invariant
that measures the asymptotic non-dissipating content of the tuple.
\end{abstract}

\author{James Tian}
\address{Mathematical Reviews, 535 W. William St, Suite 210, Ann Arbor, MI
48103, USA}
\email{james.ftian@gmail.com}

\maketitle
\tableofcontents{}

\section{Introduction}\label{sec:1}

The analysis of non-self-adjoint operators often relies on embedding
the operator into a more symmetric space where the structure becomes
clearer. Since the foundational works of Aronszajn on reproducing
kernels \cite{MR51437} and Stinespring on positive linear maps \cite{MR69403},
it has been understood that dilation mechanisms can reveal the true
geometry of an operator. Classical models for single contractions,
refined through the Sz.-Nagy-Foias theory \cite{MR2760647}, provided
the first major example (see also \cite{MR2743416} for a modern overview).

The multivariable setting has been shaped by Frazho \cite{MR671311},
Bunce \cite{MR744917}, and Popescu \cite{MR972704}, whose introduction
of noncommuting models and isometric dilations for operator tuples
established a new paradigm for free semigroup operator theory. This
line of work expanded rapidly, influencing free probability \cite{MR799593,MR1094052,MR1407898,MR2266879},
noncommutative analytic Toeplitz algebras \cite{MR1625750}, and the
function-theoretic aspects of multivariable operator theory \cite{MR1129595,MR2566311,MR3345180}.
This area remains highly active, with recent developments extending
functional models and dilation techniques to tetrablock contractions
\cite{MR4534539}, commuting Hilbert-space contractions \cite{MR4200246},
$q$-commuting co-extensions \cite{MR4510089}, and pairs of commuting
pure contractions \cite{MR4804327}. Further refinements include,
for example, $\Gamma_{n}$-contractions \cite{MR4369650,MR4765823},
negative powers of contractions \cite{MR4720521}, and absolutely
dilatable Schur multipliers \cite{MR4690505}.

The present paper approaches dilation theory through the structure
of non-commutative moment kernels. Such kernels arise naturally whenever
one records the correlations of operator words along the free semigroup.
While these objects are often motivated by stochastic systems, such
as random Carleson sequences \cite{MR4796153}, random kernel approximation
\cite{MR4896090}, and stochastic operator models \cite{MR4595378},
they are interesting in their own right as self-contained operator-valued
kernels. This perspective aligns with the rich contemporary literature
on structured function spaces, including generalized de Branges-Rovnyak
spaces \cite{MR4866510,MR4567339}, Dirichlet spaces \cite{MR3906294},
and vector-valued de Branges spaces \cite{MR4507623,MR4513110}. Parallel
developments in sampling formulas \cite{MR3711878}, interpolating
sequences \cite{MR4605405}, Parseval frames \cite{MR4860987}, and
weighted Bergman kernels \cite{MR3522008} show the utility of kernel
methods in capturing subtle analytic phenomena. Our work situates
moment kernels within this context, leveraging connections to smooth
approximations \cite{MR4521737}, contractive multipliers \cite{MR3626500},
polyanalytic functions \cite{MR4658229}, and Hilbert space singularities
\cite{MR4530898}.

Our first main observation is that a single shift-positivity condition
on the moment kernel plays the exact role that contractivity plays
in classical dilation theory. When this condition holds, the kernel
dilates to a Cuntz representation, and all of its correlations are
realized by compressing the universal model. This yields a kernel-theoretic
analogue of Popescu's dilation theorem \cite{MR972704}. The resulting
dilation does not concern the underlying operators that may or may
not have generated the kernel; it is instead a structural property
of the kernel itself. Related dilation mechanisms appear in the work
of Curgus-Dijksma \cite{MR4644887}, Kuzel \cite{MR4751521}, and
Bhat-Chongdar \cite{MR4802624}.

A free functional calculus then naturally follows. In analogy with
the free disk and Hardy algebra calculi \cite{MR1129595,MR2566311},
evaluation of polynomials with respect to the kernel extends uniquely
and contractively to the free analytic algebras. This identifies an
intrinsic Banach space associated with the moment kernel and provides
access to higher-order analytic properties of the corresponding model.
Recent advances in analytic and geometric operator theory, such as
those of Lu-Yang-Zu on compressed shifts \cite{MR4604840}, Ball-Sau
on functional models \cite{MR4200246}, and the structural work on
semigroup-generated isometries \cite{MR4460482,MR4430950}, demonstrate
the usefulness of such calculi. Furthermore, recent work on analytic
Besov calculi \cite{MR4757479} and characterizations of dilations
via approximants \cite{MR4622471} suggests that these tools can be
adapted to more general function spaces.

Next, we extract a multi-level defect hierarchy for moment kernels.
Classical dilation theory typically involves a single defect operator
whose positivity determines the existence of a dilation. By contrast,
moment kernels naturally carry an infinite sequence of higher-order
defects. We show that positivity of the entire chain is equivalent
to the existence of a single Cuntz dilation equipped with a nested
sequence of projections encoding all defect levels simultaneously.
This structural result parallels aspects of the Wold decomposition
for isometries and row-isometries \cite{MR4651635,MR3151275}, $\mathcal{U}_{n}$-twisted
contractions \cite{MR4555751}, and odometer semigroups \cite{MR4430950}.
It also connects to recent refinements in block decompositions of
moment dilations \cite{MR4874969} and complete Nevanlinna-Pick kernels
\cite{MR4592266}, but arises here solely from kernel data.

With this defect hierarchy in place, we obtain a Wold-type decomposition
for moment kernels. Every contractive moment kernel splits uniquely
into a shift-generated part and a purely dissipative part. The decomposition
is realized inside the minimal Cuntz dilation and identifies exactly
which components of the kernel behave ``unitarily'' and which decay
under the shift. This echoes recent works on decay and unitary asymptotes
\cite{MR4658745}, shift-type invariant subspaces \cite{MR4396935},
and Brownian extensions \cite{MR4617048}.

Finally, we introduce a curvature-type invariant for moment kernels.
Inspired by the curvature invariants of Arveson \cite{MR1712636,MR1758582}
and their noncommutative analogues \cite{MR1822685,MR1857801,MR3345180},
the invariant measures the asymptotic behavior of the shift action
on the kernel. It admits several equivalent formulations, including
a defect-summation identity reminiscent of trace identities in noncommutative
martingale theory \cite{MR1916654,MR2319715,MR2338860,MR2589944}.
The invariant is stable under compression, dilation, and unitary similarity
\cite{MR2891735}, and thus provides a quantitative descriptor of
the kernel's long-term structure.

\section{Dilation of Free Moment Kernels}\label{sec:2}

The starting point of our analysis is the moment kernel associated
with a tuple of random operators. In classical dilation theory, one
studies a fixed operator tuple and encodes its algebraic relations
through powers and products. In the random setting, the natural object
is obtained by averaging such products, leading to a positive definite
kernel on the free semigroup. This kernel captures the mean-square
behavior of the random tuple and provides the right entry point for
dilation. The goal of this section is to characterize precisely when
such kernels admit dilations to Cuntz families of isometries, and
to establish the equivalence between this dilation property and a
positivity condition on a shifted version of the kernel.

Fix $d\in\mathbb{N}$. Let $\mathbb{F}{}^{+}_{d}$ be the free semigroup
on generators $\left\{ 1,\dots,d\right\} $, with neutral element
$\emptyset$. For a word $\alpha=i_{1}\dots i_{k}\in\mathbb{F}^{+}_{d}$,
write $\tilde{\alpha}=i_{k}\cdots i_{1}$ for the reversed word. Given
any $d$-tuple of operators $S=\left(S_{1},\dots,S_{d}\right)$ on
a Hilbert space $H$, define 
\[
S^{\alpha}\coloneqq S_{i_{1}}\cdots S_{i_{k}},\quad S^{\emptyset}\coloneqq I_{H}.
\]
Then for every $\alpha\in\mathbb{F}^{+}_{d}$, 
\[
\left(S^{\alpha}\right)^{*}=\left(S^{*}\right)^{\tilde{\alpha}},
\]
where $S^{*}=\left(S^{*}_{1},\dots,S^{*}_{d}\right)$.

Let $\left(\Omega,\mathbb{P}\right)$ be a probability space and $A=\left(A_{1},\dots,A_{d}\right)$
a $d$-tuple of (strongly measurable) random operators $A_{i}\colon\Omega\rightarrow\mathcal{L}\left(H\right)$
on a Hilbert space $H$, and assume 
\[
\mathbb{E}\left[\left\Vert A^{\alpha}\right\Vert ^{2}\right]<\infty,\quad\alpha\in\mathbb{F}^{+}_{d}.
\]

Define the moment kernel $K\colon\mathbb{F}^{+}_{d}\times\mathbb{F}^{+}_{d}\rightarrow\mathcal{L}\left(H\right)$
by 
\begin{equation}
K\left(\alpha,\beta\right)\coloneqq\mathbb{E}\left[A^{\alpha}\left(A^{\beta}\right)^{*}\right],\quad\alpha,\beta\in\mathbb{F}^{+}_{d}.\label{eq:a-1}
\end{equation}
Then $K$ is positive definite (p.d.): for every finite family $\left\{ \left(\alpha_{j},u_{j}\right)\right\} ^{N}_{j=1}\subset\mathbb{F}^{+}_{d}\times H$,
\begin{align*}
\sum\nolimits^{N}_{i,j=1}\left\langle u_{i},K\left(\alpha_{i},\alpha_{j}\right)u_{j}\right\rangle _{H} & =\sum\nolimits^{N}_{i,j=1}\mathbb{E}\left[\left\langle u_{i},A^{\alpha_{i}}\left(A^{\alpha_{j}}\right)^{*}u_{j}\right\rangle _{H}\right]\\
 & =\mathbb{E}\left[\left\Vert \sum\nolimits_{i}\left(A^{\alpha_{i}}\right)^{*}u_{i}\right\Vert ^{2}_{H}\right]\geq0.
\end{align*}
Define the kernel 
\begin{equation}
K_{\Sigma}\left(\alpha,\beta\right)\coloneqq\sum^{d}_{i=1}K\left(\alpha i,\beta i\right),\quad\alpha,\beta\in\mathbb{F}^{+}_{d}.\label{eq:b-2}
\end{equation}

\begin{thm}[Free moment dilation]
\label{thm:1}For the moment kernel $K$ in \eqref{eq:a-1}, the
following are equivalent: 
\begin{enumerate}
\item \label{enu:1a}There exists a Hilbert space $\mathcal{K}$, a Cuntz
family $S=\left(S_{1},\dots,S_{d}\right)$ on $\mathcal{K}$ (i.e.,
$S^{*}_{i}S_{j}=\delta_{i,j}I_{\mathcal{K}}$ and $\sum^{d}_{i=1}S_{i}S^{*}_{i}=I_{\mathcal{K}}$),
a projection $P\in\mathcal{L}\left(\mathcal{K}\right)$, and an isometry
$W\colon H\rightarrow\mathcal{K}$ such that 
\begin{equation}
K\left(\alpha,\beta\right)=W^{*}S^{\alpha}P\left(S^{\beta}\right)^{*}W,\quad\alpha,\beta\in\mathbb{F}^{+}_{d};\label{eq:a1}
\end{equation}
and, 
\begin{equation}
\sum\nolimits^{d}_{i=1}S_{i}PS^{*}_{i}\leq P\;\text{on }\mathcal{M}\coloneqq\overline{span}\left\{ \left(S^{\gamma}\right)^{*}WH:\gamma\in\mathbb{F}^{+}_{d}\right\} \label{eq:a2}
\end{equation}
\item \label{enu:1b}The kernel $K_{\Sigma}$ is dominated by $K$ in the
p.d. order: 
\begin{equation}
K_{\Sigma}\leq K.\label{eq:a3}
\end{equation}
(That is, $K-K_{\Sigma}$ is p.d.) 
\end{enumerate}
Moreover, $K_{\Sigma}=K$ if and only if 
\begin{equation}
K\left(\alpha,\beta\right)=W^{*}S^{\alpha}\left(S^{\beta}\right)^{*}W,\quad\alpha,\beta\in\mathbb{F}^{+}_{d}.\label{eq:a-5}
\end{equation}

\end{thm}

\begin{proof}
\eqref{enu:1a} $\Rightarrow$ \eqref{enu:1b}. Fix a finite family
$\left\{ \left(\alpha_{j},u_{j}\right)\right\} ^{N}_{j=1}\subset\mathbb{F}^{+}_{d}\times H$
and set 
\[
y\coloneqq\sum\nolimits^{N}_{j=1}\left(S^{\alpha_{j}}\right)^{*}Wu_{j}\in M.
\]
Then (using $S^{\alpha i}=S^{\alpha}S_{i}$ and \eqref{eq:a1}) 
\begin{eqnarray*}
 &  & \sum\nolimits^{N}_{j,k=1}\left\langle u_{j},K_{\Sigma}\left(\alpha_{j},\alpha_{k}\right)u_{k}\right\rangle _{H}\\
 & = & \sum\nolimits^{d}_{i=1}\sum\nolimits^{N}_{j,k=1}\left\langle u_{j},K\left(\alpha_{j}i,\alpha_{k}i\right)u_{k}\right\rangle _{H}\\
 & = & \sum\nolimits^{d}_{i=1}\sum\nolimits^{N}_{j,k=1}\left\langle \left(S^{\alpha_{j}}\right)^{*}Wu_{j},\left(S_{i}PS^{*}_{i}\right)\left(S^{\alpha_{k}}\right)^{*}Wu_{k}\right\rangle _{\mathcal{K}}\\
 & = & \left\langle y,\left(\sum\nolimits^{d}_{i=1}S_{i}PS^{*}_{i}\right)y\right\rangle _{\mathcal{K}}.
\end{eqnarray*}
Similarly, using \eqref{eq:a1}, 
\[
\sum\nolimits^{N}_{j,k=1}\left\langle u_{j},K\left(\alpha_{j},\alpha_{k}\right)u_{k}\right\rangle _{H}=\sum\nolimits^{N}_{j,k=1}\left\langle \left(S^{\alpha_{j}}\right)^{*}Wu_{j},P\left(S^{\alpha_{k}}\right)^{*}Wu_{k}\right\rangle _{\mathcal{K}}=\left\langle y,Py\right\rangle _{\mathcal{K}}.
\]
By assumption \eqref{eq:a2}, we have 
\[
\left\langle y,\left(\sum\nolimits^{d}_{i=1}S_{i}PS^{*}_{i}\right)y\right\rangle _{\mathcal{K}}\leq\left\langle y,Py\right\rangle _{\mathcal{K}},
\]
which is exactly \eqref{eq:a3}.

\eqref{enu:1b} $\Rightarrow$ \eqref{enu:1a}. Since $K$ is p.d.,
it admits a Kolmogorov/GNS factorization: there exists a Hilbert space
$H'$ and operators $V_{\alpha}\colon H\rightarrow H'$ such that
\[
K\left(\alpha,\beta\right)=V^{*}_{\alpha}V_{\beta},\quad\alpha,\beta\in\mathbb{F}^{+}_{d},
\]
and $\mathscr{D}\coloneqq span\left\{ V_{\alpha}u:\alpha\in\mathbb{F}^{+}_{d},u\in H\right\} $
is dense in $H'$. 

Concretely, one may take $H'$ to be the reproducing kernel Hilbert
space (RKHS) of the scalar valued kernel 
\[
\tilde{K}\left(\left(\alpha,u\right),\left(\beta,v\right)\right)\coloneqq\left\langle u,K\left(\alpha,\beta\right)v\right\rangle _{H},
\]
defined on $\left(\mathbb{F}^{+}_{d}\times H\right)\times\left(\mathbb{F}^{+}_{d}\times H\right)$,
and set 
\[
V_{\alpha}u=\tilde{K}\left(\cdot,\left(\alpha,u\right)\right)\in H'.
\]
For details, see e.g., \cite{MR2938971,MR4250453,tian2025randomoperatorvaluedkernelsmoment}. 

On the algebraic span $\mathscr{D}$, define 
\[
B_{i}\left(V_{\alpha}u\right)\coloneqq V_{\alpha i}u,\quad i=1,\dots,d,
\]
and extend linearly. 

Take an arbitrary element $x=\sum_{\alpha}V_{\alpha}u_{\alpha}\in\mathscr{D}$.
Then 
\[
\left\Vert x\right\Vert ^{2}_{H'}=\sum\nolimits_{\alpha,\beta}\left\langle u_{\alpha},K\left(\alpha,\beta\right)u_{\beta}\right\rangle _{H},\quad\left\Vert B_{i}x\right\Vert ^{2}_{H'}=\sum\nolimits_{\alpha,\beta}\left\langle u_{\alpha},K\left(\alpha i,\beta i\right)u_{\beta}\right\rangle _{H}.
\]
Summing over $i$, 
\begin{align*}
\sum\nolimits^{d}_{i=1}\left\Vert B_{i}x\right\Vert ^{2}_{H'} & =\sum\nolimits_{\alpha,\beta}\left\langle u_{\alpha},\left(\sum\nolimits^{d}_{i=1}K\left(\alpha i,\beta i\right)\right)u_{\beta}\right\rangle _{H}\\
 & =\sum\nolimits_{\alpha,\beta}\left\langle u_{\alpha},K_{\Sigma}\left(\alpha,\beta\right)u_{\beta}\right\rangle _{H}\\
 & \leq\sum\nolimits_{\alpha,\beta}\left\langle u_{\alpha},K\left(\alpha,\beta\right)u_{\beta}\right\rangle _{H}=\left\Vert x\right\Vert ^{2}_{H'},
\end{align*}
by the hypothesis $K_{\Sigma}\leq K$. Hence the linear map 
\[
B_{col}\colon\mathscr{D}\rightarrow H'^{\oplus d},\quad B_{col}x\coloneqq\left(B_{1}x,\dots,B_{d}x\right)
\]
is contractive on $\mathscr{D}$; it extends by continuity to a contraction
$B_{col}\colon H'\rightarrow H'^{\oplus d}$, and each $B_{i}$ has
a bounded extension to $H'$ with 
\[
\sum B^{*}_{i}B_{i}\leq I_{H'}\quad\left(\text{as operators on }H'\right).
\]
(So $B_{col}$ is a column contraction; equivalently the adjoint $\left(B^{*}_{1},\dots,B^{*}_{d}\right)$
is a row contraction.)

By the multivariable dilation theory of Frazho, Bunce, and Popescu
(\cite{MR671311,MR744917,MR972704}), there exists a Hilbert space
$\mathcal{K}$, a Cuntz family $S=\left(S_{1},\dots,S_{d}\right)$
on $\mathcal{K}$ with orthogonal ranges, and an isometry $J\colon H'\rightarrow\mathcal{K}$
such that 
\[
S^{*}_{i}J=JB_{i},\quad i=1,\dots,d.
\]
In particular, for every word $\alpha\in\mathbb{F}^{+}_{d}$, 
\[
B^{\alpha}=J^{*}\left(S^{\tilde{\alpha}}\right)^{*}J.
\]
(Recall $\tilde{\alpha}$ is the reversed word for $\alpha\in\mathbb{F}^{+}_{d}$.)

Note that $V_{\alpha}=B^{\tilde{\alpha}}V_{\emptyset}$ (by construction).
Therefore, 

\begin{align*}
K\left(\alpha,\beta\right) & =V^{*}_{\alpha}V_{\beta}=V^{*}_{\emptyset}\left(B^{\tilde{\alpha}}\right)^{*}B^{\tilde{\beta}}V_{\emptyset}\\
 & =V^{*}_{\emptyset}\left(J^{*}\left(S^{\alpha}\right)^{*}J\right)^{*}\cdot J^{*}\left(S^{\beta}\right)^{*}JV_{\emptyset}\\
 & =\left(JV_{\emptyset}\right)^{*}S^{\alpha}\left(JJ^{*}\right)\left(S^{\beta}\right)^{*}\left(JV_{\emptyset}\right)\\
 & =W^{*}S^{\alpha}P\left(S^{\beta}\right)^{*}W
\end{align*}
where $W\coloneqq JV_{\emptyset}:H\rightarrow\mathcal{K}$ (isometric
because $J$ and $V_{\emptyset}$ are), and $P\coloneqq JJ^{*}$ (the
orthogonal projection onto $JH'\subset\mathcal{K}$). 

To verify \eqref{eq:a2}, consider any $y=\sum_{j}\left(S^{\alpha_{j}}\right)^{*}Wu_{j}\in M$,
\[
\left\langle y,S_{i}PS^{*}_{i}y\right\rangle _{\mathcal{K}}=\left\langle y,S_{i}JJ^{*}S^{*}_{i}y\right\rangle _{\mathcal{K}}=\left\Vert J^{*}S^{*}_{i}y\right\Vert ^{2}_{H'}
\]
and $\left\langle y,Py\right\rangle _{\mathcal{K}}=\left\Vert J^{*}y\right\Vert ^{2}_{H'}$.
It suffices to show 
\begin{equation}
\sum\nolimits^{d}_{i=1}\left\Vert J^{*}S^{*}_{i}y\right\Vert ^{2}_{H'}\leq\left\Vert J^{*}y\right\Vert ^{2}_{H'}.\label{eq:a4}
\end{equation}
But, 
\[
J^{*}y=\sum\nolimits_{j}J^{*}\left(S^{\alpha_{j}}\right)^{*}Wu_{j}=\sum\nolimits_{j}J^{*}\left(S^{\alpha_{j}}\right)^{*}JV_{\emptyset}u_{j}=\sum\nolimits_{j}B^{\tilde{\alpha}_{j}}V_{\emptyset}u_{j},
\]
and 
\[
J^{*}S^{*}_{i}y=\sum\nolimits_{j}J^{*}\left(S^{\alpha_{j}i}\right)^{*}JV_{\emptyset}u_{j}=\sum\nolimits_{j}B^{i\tilde{\alpha_{j}}}V_{\emptyset}u_{j}=B_{i}\left(\sum\nolimits_{j}B^{\tilde{\alpha}_{j}}V_{\emptyset}u_{j}\right)=B_{i}J^{*}y.
\]
Thus, \eqref{eq:a4} is equivalent to 
\[
\sum\nolimits^{d}_{i=1}\left\Vert B_{i}J^{*}y\right\Vert ^{2}_{H'}\leq\left\Vert J^{*}y\right\Vert ^{2}_{H'},
\]
which holds since $\sum B^{*}_{i}B_{i}\leq I_{H'}$. We conclude that
\eqref{eq:a2} holds. 

Lastly, if $K$ is of the form $K\left(\alpha,\beta\right)=W^{*}S^{\alpha}\left(S^{\beta}\right)^{*}W$,
then 
\begin{align*}
K_{\Sigma}\left(\alpha,\beta\right) & =\sum\nolimits_{i}K\left(\alpha i,\beta i\right)\\
 & =\sum\nolimits_{i}W^{*}S^{\alpha i}\left(S^{\beta i}\right)^{*}W=W^{*}S^{\alpha}\left(\sum\nolimits_{i}S_{i}S^{*}_{i}\right)\left(S^{\beta}\right)^{*}W=K\left(\alpha,\beta\right).
\end{align*}

Conversely, assume $K_{\Sigma}=K$. From the discussion above, it
follows that 
\[
\sum\nolimits^{d}_{i=1}S_{i}PS^{*}_{i}=P\;\text{on }\mathcal{M}
\]
where $M=\overline{span}\left\{ \left(S^{\gamma}\right)^{*}WH:\gamma\in\mathbb{F}^{+}_{d}\right\} $,
and $P$ is the projection from $\mathcal{K}$ onto $JH'$. Left-multiplying
the identity by $S^{*}_{j}$ and using $S^{*}_{j}S_{i}=\delta_{i,j}$,
we get $PS^{*}_{j}=S^{*}_{j}P$ on $M$ for all $j\in\left\{ 1,\dots,d\right\} $,
and so 
\[
P\left(S^{\beta}\right)^{*}=\left(S^{\beta}\right)^{*}P\;\text{on }\mathcal{M}
\]
for all $\beta\in\mathbb{F}^{+}_{d}$. In particular, $P\left(S^{\beta}\right)^{*}W=\left(S^{\beta}\right)^{*}PW=\left(S^{\beta}\right)^{*}W$.
Therefore, 
\[
K\left(\alpha,\beta\right)=W^{*}S^{\alpha}P\left(S^{\beta}\right)^{*}W=W^{*}S^{\alpha}\left(S^{\beta}\right)^{*}W.
\]
 
\end{proof}
\begin{rem}[Infinite variables]
 \prettyref{thm:1} extends without essential change to the case
$d=\infty$, where the free semigroup $\mathbb{F}^{+}_{\infty}$ is
generated by countably many letters. One defines 
\[
K_{\Sigma}\left(\alpha,\beta\right)\coloneqq\sum^{\infty}_{i=1}K\left(\alpha i,\beta i\right),
\]
with the series converging in the quadratic-form sense, and interprets
all operator sums in the strong operator topology. The Kolmogorov
construction, the definition of the operators $B_{i}$, and the dilation
result remains valid for countable tuples. In this setting one obtains
the same equivalence as in \prettyref{thm:1}. Thus the free moment
dilation theorem continues to hold for infinitely many noncommuting
variables.
\end{rem}

\begin{rem}[Connection with the one-variable case]
 When $d=1$, the free semigroup $\mathbb{F}^{+}_{1}$ is naturally
identified with $\mathbb{N}$, and words $\alpha$ correspond to integers
$k$. In this case the kernel $K_{\Sigma}$ reduces to the shifted
kernel
\[
K_{\Sigma}\left(m,n\right)=K\left(m+1,n+1\right),\qquad m,n\in\mathbb{N},
\]
and \prettyref{thm:1} recovers \cite[Theorem 8]{tian2025randomoperatorvaluedkernelsmoment}.
In that single-variable setting, the dilation involves a unitary $U$
and has the form
\[
K\left(m,n\right)=W^{*}U^{m}PU^{*n}W,
\]
with $P$ a projection. 

The projection $P$ in \prettyref{thm:1} plays the same role in the
multivariable case. It encodes the defect between the universal Cuntz
dilation and the smaller subspace that actually realizes the kernel.
In general one only has \eqref{eq:a1} with \eqref{eq:a2}. However,
in the special case $K_{\Sigma}=K$, the projection $P$ acts as the
identity on the cyclic subspace generated by $\{\left(S^{\beta}\right)^{*}WH\}$.
On this subspace the compression disappears, so the realization formula
simplifies to \eqref{eq:a-5}. Thus $P$ is not necessarily equal
to $I_{\mathcal{K}}$, but it becomes invisible in the part of the
dilation that supports the kernel.
\end{rem}

The dilation theorem allows us to pass from random operators to Cuntz
families, where classical operator inequalities are available. In
particular, once the kernel is realized through a Cuntz dilation,
the noncommutative von Neumann inequality of Popescu applies. This
yields a mean-square analogue for random operator tuples, formulated
in the following corollary.
\begin{cor}[Free mean-square von Neumann inequality]
\label{cor:4}Assume $K_{\Sigma}\leq K$ so that the dilation above
exists. For any noncommutative polynomial $f\left(Z\right)=\sum_{\alpha\in\mathbb{F}^{+}_{d}}c_{\alpha}Z^{\alpha}$
and any $u\in H$, 
\[
\mathbb{E}\left[\left\Vert f\left(A\right)^{*}u\right\Vert ^{2}\right]\le\left\Vert f\left(S\right)\right\Vert ^{2}\left\Vert u\right\Vert ^{2}\leq\left\Vert f\left(L\right)\right\Vert ^{2}\left\Vert u\right\Vert ^{2},
\]
where $S=\left(S_{1},\dots,S_{d}\right)$ is the Cuntz family from
\prettyref{thm:1}, and $L=\left(L_{1},\dots,L_{d}\right)$ denotes
the left creation operators on the full Fock space $l^{2}\left(\mathbb{F}^{+}_{d}\right)$.
Equivalently, 
\[
\mathbb{E}\left[f\left(A\right)f\left(A\right)^{*}\right]\leq\left\Vert f\left(L\right)\right\Vert ^{2}I_{H}.
\]
\end{cor}

\begin{proof}
Expanding $f\left(A\right)^{*}$, we have 
\begin{align*}
\mathbb{E}\left[\left\Vert f\left(A\right)^{*}u\right\Vert ^{2}\right] & =\sum_{\alpha,\beta}c_{\alpha}\overline{c_{\beta}}\mathbb{E}\left[\left\langle u,A^{\alpha}\left(A^{\beta}\right)^{*}u\right\rangle \right]\\
 & =\sum_{\alpha,\beta}c_{\alpha}\overline{c_{\beta}}\left\langle u,K\left(\alpha,\beta\right)u\right\rangle \\
 & =\sum_{\alpha,\beta}c_{\alpha}\overline{c_{\beta}}\left\langle u,W^{*}S^{\alpha}P\left(S^{\beta}\right)^{*}Wu\right\rangle \\
 & =\left\Vert Pf\left(S\right)^{*}Wu\right\Vert ^{2},
\end{align*}
using the representation of $K$ in \eqref{eq:a1}. Since $W$ is
an isometry and $\left\Vert Px\right\Vert \leq\left\Vert x\right\Vert $,
this is bounded by $\left\Vert f\left(S\right)\right\Vert ^{2}\left\Vert u\right\Vert ^{2}$.
Finally, Popescu's noncommutative von Neumann inequality gives the
claim, see \cite{MR1129595}. 
\end{proof}

\section{Mean-square Free Functional Calculus}\label{sec:3}

The dilation result and the mean-square von Neumann inequality from
the previous section provide the basic ingredients for building a
functional calculus for random operators. In the deterministic setting,
Popescu showed how dilations lead to contractive extensions of polynomial
evaluation, first to the free disk algebra and then to the free Hardy
algebra. Our goal here is to establish an analogous construction in
the mean-square setting. We begin by defining the appropriate Banach
space of random operators equipped with the mean-square norm, and
then show how the dilation machinery gives rise to a canonical extension
of polynomial evaluation.

\textbf{Setting.} Fix $d\in\mathbb{N}$. Let $A=\left(A_{1},\dots,A_{d}\right)$
be a $d$-tuple of strongly measurable random operators on a Hilbert
space $H$, such that 
\[
\mathbb{E}\left[\left\Vert A^{\alpha}\right\Vert ^{2}\right]<\infty,\quad\alpha\in\mathbb{F}^{+}_{d}.
\]
Define the moment kernel
\[
K\left(\alpha,\beta\right)\coloneqq\mathbb{E}\left[A^{\alpha}\left(A^{\beta}\right)^{*}\right],\quad\alpha,\beta\in\mathbb{F}^{+}_{d},
\]
and assume $K_{\Sigma}\le K$, where $K_{\Sigma}\left(\alpha,\beta\right)\coloneqq\sum^{d}_{i=1}K\left(\alpha i,\beta i\right)$.
Then the dilation in \prettyref{thm:1} exists, and \prettyref{cor:4}
holds:
\[
\mathbb{E}\left[f\left(A\right)f\left(A\right)^{*}\right]\le\left\Vert f(L)\right\Vert ^{2}I_{H}
\]
for every noncommutative polynomial $f$, where $L=\left(L_{1},\dots,L_{d}\right)$
are the left creation operators on $\ell^{2}\left(\mathbb{F}^{+}_{d}\right)$.
\begin{defn}
For a strongly measurable random operator $X\colon\Omega\rightarrow\mathcal{L}\left(H\right)$,
define the mean-square seminorm 
\[
\left\Vert X\right\Vert _{ms}\coloneqq\Big\|\mathbb{E}\left[XX^{*}\right]\Big\|^{1/2}=\sup_{\|u\|=1}\left(\mathbb{E}\left[\left\Vert X^{*}u\right\Vert ^{2}\right]\right)^{1/2}.
\]
Set 
\[
\mathcal{L}_{ms}\left(\Omega,H\right)\coloneqq\left\{ X\colon\Omega\rightarrow\mathcal{L}\left(H\right)\mid\left\Vert X\right\Vert _{ms}<\infty\right\} /\sim
\]
where $X\sim Y$ if $\left\Vert X-Y\right\Vert _{ms}=0$. Then $\mathcal{L}_{ms}\left(\Omega,H\right)$
is a Banach space of random operators with an isometric involution
$X\mapsto X^{*}$.
\end{defn}

\begin{rem}
If $X\left(\omega\right)\equiv T$ is deterministic, then 
\[
\left\Vert X\right\Vert _{ms}=\left\Vert TT^{*}\right\Vert ^{1/2}=\left\Vert T\right\Vert .
\]
Thus, in the non-random case the mean-square norm reduces to the usual
operator norm on $\mathcal{L}\left(H\right)$.

This mean-square space is a special case of the conditioned column
spaces $L^{c}_{\infty}\left(M,E\right)$ associated to a conditional
expectation $E$ on a von Neumann algebra $M$ in the sense of Junge
and collaborators \cite{MR1916654,MR2338860,MR2589944}; see also
Randrianantoanina's work on conditioned square functions for noncommutative
martingales \cite{MR2319715}.
\end{rem}

\begin{defn}
\label{def:7}A net $\left(X_{\lambda}\right)$ of random operators
on $H$ converges to $X$ in ms-SOT if 
\[
\mathbb{E}\left\Vert \left(X_{\lambda}-X\right)^{*}u\right\Vert ^{2}\rightarrow0,\quad\forall u\in H.
\]
We then write $X_{\lambda}\xrightarrow{\text{ms-SOT}}X$. 
\end{defn}

\begin{thm}[Mean-square free disk algebra calculus]
\label{thm:7}Let $\mathbb{C}\langle Z\rangle=\mathbb{C}\left\langle Z_{1},\dots,Z_{d}\right\rangle $
be the algebra of noncommutative polynomials and endow it with the
Fock-space norm
\[
\left\Vert p\right\Vert _{L}\coloneqq\left\Vert p(L)\right\Vert ,\quad p\in\mathbb{C}\langle Z\rangle.
\]
Let $\mathcal{A}_{d}$ be the norm completion of $\mathbb{C}\langle Z\rangle$
under $\left\Vert \cdot\right\Vert _{L}$ (the free disk algebra). 

Then there exists a unique linear map
\[
\Phi\colon\mathcal{A}_{d}\longrightarrow\mathcal{L}_{ms}\left(\Omega,H\right)
\]
with the following properties:
\end{thm}

\begin{enumerate}
\item (Polynomial agreement) For every polynomial $p\in\mathbb{C}\langle Z\rangle$,
one has $\Phi\left(p\right)=p\left(A\right)$ (as a random operator),
and
\[
\left\Vert \Phi\left(p\right)\right\Vert _{ms}\le\left\Vert p\left(L\right)\right\Vert .
\]
\item (Continuity) For all $f\in\mathcal{A}_{d}$,
\[
\left\Vert \Phi\left(f\right)\right\Vert _{ms}\leq\left\Vert f\left(L\right)\right\Vert .
\]
In particular, if $f_{n}\to f$ in $\left\Vert \cdot\right\Vert _{L}$,
then $\left\Vert \Phi\left(f_{n}\right)-\Phi\left(f\right)\right\Vert _{ms}\rightarrow0$.
\item (Uniqueness) $\Phi$ is the only linear map $\mathcal{A}_{d}\to\mathcal{L}_{ms}\left(\Omega,H\right)$
that is continuous for $\left\Vert \cdot\right\Vert _{L}$ and agrees
with polynomial evaluation.
\item (Strong mean-square convergence) If $f_{n}\to f$ in $\left\Vert \cdot\right\Vert _{L}$,
then $\Phi\left(f_{n}\right)\xrightarrow{\text{ms-SOT}}\Phi\left(f\right)$.
That is, for every $u\in H$,
\[
\mathbb{E}\left\Vert \left(\Phi\left(f_{n}\right)\right)^{*}u-\left(\Phi\left(f\right)\right)^{*}u\right\Vert ^{2}\longrightarrow0.
\]
\end{enumerate}
\begin{proof}
Let $p,q$ be noncommutative polynomials. Apply \prettyref{cor:4}
to the polynomial $p-q$:
\begin{equation}
\mathbb{E}\left[\left(p\left(A\right)-q\left(A\right)\right)\left(p\left(A\right)-q\left(A\right)\right)^{*}\right]\le\left\Vert \left(p-q\right)\left(L\right)\right\Vert ^{2}I_{H}.\label{eq:c1}
\end{equation}
Taking the operator norm and square-root gives
\begin{equation}
\left\Vert p\left(A\right)-q\left(A\right)\right\Vert _{ms}\le\left\Vert p\left(L\right)-q\left(L\right)\right\Vert .\label{eq:c2}
\end{equation}
In particular, $\left\Vert p\left(A\right)\right\Vert _{ms}\leq\left\Vert p\left(L\right)\right\Vert $. 

Fix $f\in\mathcal{A}_{d}$. By definition of $\mathcal{A}_{d}$, choose
polynomials $p_{n}$ with $\left\Vert p_{n}\left(L\right)-f\left(L\right)\right\Vert \to0$.
By \eqref{eq:c2}, $\left\{ p_{n}\left(A\right)\right\} $ is Cauchy
in $\left\Vert \cdot\right\Vert _{ms}$. Hence there exists a (mean-square)
limit in $\mathcal{L}_{ms}\left(\Omega,H\right)$ which we denote
by $\Phi\left(f\right)$:
\[
\left\Vert p_{n}\left(A\right)-\Phi\left(f\right)\right\Vert _{ms}\longrightarrow0.
\]
For polynomials $p$, clearly $\Phi\left(p\right)=p\left(A\right)$.
This defines $\Phi$ on $\mathcal{A}_{d}$.

\textit{Independence of approximants}. Suppose also $q_{n}$ are polynomials
with 
\[
\left\Vert q_{n}\left(L\right)-f\left(L\right)\right\Vert \to0.
\]
Then $\left\Vert p_{n}\left(L\right)-q_{n}\left(L\right)\right\Vert \to0$,
so by \eqref{eq:c2}, 
\[
\left\Vert p_{n}\left(A\right)-q_{n}\left(A\right)\right\Vert _{ms}\to0.
\]
Hence $\lim p_{n}\left(A\right)=\lim q_{n}\left(A\right)$ in mean
square. Thus $\Phi\left(f\right)$ is independent of the chosen approximating
sequence.

\textit{Linearity}. Take $f,g\in\mathcal{A}_{d}$ and scalars $a,b$.
Let $p_{n}\to f$, $q_{n}\to g$ in $\left\Vert \cdot\right\Vert _{L}$.
Then
\[
\left\Vert \left(ap_{n}+bq_{n}\right)\left(L\right)-\left(af+bg\right)\left(L\right)\right\Vert \to0,
\]
so $\left(ap_{n}+bq_{n}\right)\left(A\right)\to\Phi\left(af+bg\right)$
in mean square. But also $\left(ap_{n}+bq_{n}\right)\left(A\right)=ap_{n}\left(A\right)+bq_{n}\left(A\right)\to a\Phi\left(f\right)+b\Phi\left(g\right)$.
Hence 
\[
\Phi\left(af+bg\right)=a\Phi\left(f\right)+b\Phi\left(g\right).
\]

\textit{Contractivity and continuity}. Let $f\in\mathcal{A}_{d}$
and $p_{n}\to f$ in $\left\Vert \cdot\right\Vert _{L}$. Then
\[
\left\Vert \Phi\left(f\right)\right\Vert _{ms}\le\liminf_{n\to\infty}\left\Vert p_{n}\left(A\right)\right\Vert _{ms}\le\lim_{n\to\infty}\left\Vert p_{n}\left(L\right)\right\Vert =\left\Vert f\left(L\right)\right\Vert .
\]
Similarly, for $f,g\in\mathcal{A}_{d}$ and approximants $p_{n}\to f$,
$q_{n}\to g$,
\begin{align*}
\left\Vert \Phi\left(f\right)-\Phi\left(g\right)\right\Vert _{ms} & \le\liminf_{n}\left\Vert p_{n}\left(A\right)-q_{n}\left(A\right)\right\Vert _{ms}\\
 & \leq\lim_{n}\left\Vert p_{n}\left(L\right)-q_{n}\left(L\right)\right\Vert _{L}=\left\Vert f\left(L\right)-g\left(L\right)\right\Vert .
\end{align*}
Thus $\Phi$ is contractive and continuous as claimed.

\textit{Strong mean-square convergence}. Fix $u\in H$. For any noncommutative
polynomials $p,q$, by \eqref{eq:c1}, 
\begin{align*}
\mathbb{E}\left\Vert p\left(A\right)^{*}u-q\left(A\right)^{*}u\right\Vert ^{2} & =\left\langle u,\mathbb{E}\left[\left(p\left(A\right)-q\left(A\right)\right)\left(p\left(A\right)-q\left(A\right)\right)^{*}\right]u\right\rangle \\
 & \le\left\Vert p\left(L\right)-q\left(L\right)\right\Vert ^{2}\left\Vert u\right\Vert ^{2}.
\end{align*}
Hence, if $p_{n}\to f$ in $\left\Vert \cdot\right\Vert _{L}$, then
$p_{n}\left(A\right)^{*}u$ is Cauchy in $L^{2}\left(\Omega,H\right)$,
so
\[
p_{n}\left(A\right)^{*}u\longrightarrow\left(\Phi\left(f\right)\right)^{*}u\quad\text{in }L^{2}\left(\Omega,H\right).
\]
This proves the stated stated ms-SOT convergence.

\textit{Uniqueness of $\Phi$}. Finally, suppose $\Psi\colon\mathcal{A}_{d}\to\mathcal{L}_{ms}\left(\Omega,H\right)$
is another linear map, continuous for $\left\Vert \cdot\right\Vert _{L}$,
with $\Psi\left(p\right)=p\left(A\right)$ for all polynomials. For
$f\in\mathcal{A}_{d}$, take $p_{n}\to f$ in $\left\Vert \cdot\right\Vert _{L}$.
Then
\begin{align*}
\left\Vert \Psi\left(f\right)-\Phi\left(f\right)\right\Vert _{ms} & \le\left\Vert \Psi\left(f\right)-\Psi\left(p_{n}\right)\right\Vert _{ms}+\left\Vert \Psi\left(p_{n}\right)-\Phi\left(f\right)\right\Vert _{ms}\\
 & =\left\Vert \Psi\left(f\right)-\Psi\left(p_{n}\right)\right\Vert _{ms}+\left\Vert \Phi\left(p_{n}\right)-\Phi\left(f\right)\right\Vert _{ms}\longrightarrow0,
\end{align*}
since $\Psi$ is continuous and $\Psi\left(p_{n}\right)=\Phi\left(p_{n}\right)=p_{n}\left(A\right)\to\Phi\left(f\right)$.
Hence $\Psi=\Phi$, proving uniqueness.
\end{proof}
\begin{rem}
The calculus above is “canonical” in the sense that it is determined
solely by the inequality from \prettyref{cor:4} (hence by $A=\left(A_{1},\dots,A_{d}\right)$
and the condition $K_{\Sigma}\le K$), and does not depend on a particular
choice of dilation space or Cuntz family realizing $K$.
\end{rem}

It is natural to ask whether the construction extends to the full
free Hardy algebra $F^{\infty}_{d}$, the weak operator topology (WOT)-closed
algebra generated by the left creation tuple $L$ on the full Fock
space. In the deterministic setting, this extension was carried out
by Popescu using the method of radial dilations, following the classical
Sz.-Nagy-Foiaş approach to the $H^{\infty}$ calculus. See, for instance,
Popescu's work on noncommutative Hardy algebras \cite{MR2566311},
as well as the parallel operator-algebraic perspective of Davidson-Pitts
\cite{MR1625750}.

We now adapt the same idea to the mean-square setting. The proof follows
Popescu's argument, but in a probabilistic context: our functional
calculus acts on random operator tuples, takes values in the Banach
space $\mathcal{L}_{ms}(\Omega,H)$, and convergence is understood
in ms-SOT sense (\prettyref{def:7}). This provides a stochastic analogue
of Popescu's free $H^{\infty}$ calculus.
\begin{thm}[Mean-square free $H^{\infty}$ calculus]
\label{thm:10} Let $F^{\infty}_{d}=\overline{\mathcal{A}_{d}}^{WOT}$
denote the (left) free Hardy algebra (the WOT-closed algebra generated
by the left creation tuple $L$ on the full Fock space). For $\phi\in F^{\infty}_{d}$,
write its radial dilates by
\[
\phi_{r}\left(Z\right)\coloneqq\sum\nolimits_{\alpha}r^{\left|\alpha\right|}\phi_{\alpha}Z^{\alpha},\quad0<r<1,
\]
so that $\phi_{r}(L)\in\mathcal{A}_{d}$ (the free disk algebra) and
$\left\Vert \phi_{r}\left(L\right)\right\Vert \le\left\Vert \phi\left(L\right)\right\Vert $,
while $\phi_{r}\left(L\right)\to\phi\left(L\right)$ in the strong
operator topology as $r\uparrow1$.

Then there exists a unique linear map
\[
\Phi\colon F^{\infty}_{d}\longrightarrow\mathcal{L}_{ms}\left(\Omega,H\right)
\]
such that for every $\phi\in F^{\infty}_{d}$:
\begin{enumerate}
\item (Definition by radial limit) $\phi_{r}\left(A\right)\xrightarrow{\text{ms-SOT}}\Phi\left(\phi\right)$
as $r\uparrow1$, i.e., 
\[
\lim_{r\uparrow1}\mathbb{E}\left\Vert \left(\phi_{r}\left(A\right)-\Phi\left(\phi\right)\right)^{*}u\right\Vert ^{2}=0
\]
for all $u\in H$.
\item (Contractivity) We have 
\[
\mathbb{E}\left[\Phi\left(\phi\right)\Phi\left(\phi\right)^{*}\right]\le\left\Vert \phi\left(L\right)\right\Vert ^{2}I_{H},
\]
equivalently, 
\[
\sup_{\|u\|=1}\mathbb{E}\left\Vert \left(\Phi\left(\phi\right)\right)^{*}u\right\Vert ^{2}\le\left\Vert \phi\left(L\right)\right\Vert ^{2}.
\]
\item (Compatibility) If $\phi\in\mathcal{A}_{d}$, then $\Phi(\phi)$ coincides
with the mean-square limit defined in \prettyref{thm:7}.
\item (Uniqueness) $\Phi$ is the only linear map with these properties.
\end{enumerate}
\end{thm}

\begin{proof}
For any noncommutative polynomial $f$, \prettyref{cor:4} gives,
for all $u\in H$,
\begin{equation}
\mathbb{E}\left\Vert f\left(A\right)^{*}u\right\Vert ^{2}=\left\Vert P^{1/2}f\left(S\right)^{*}Wu\right\Vert ^{2}.\label{eq:c3}
\end{equation}
Using the $\mathcal{A}{}_{d}$ calculus from \prettyref{thm:7}, passing
to the limit in the polynomial identity \eqref{eq:c3}, we get 
\begin{equation}
\mathbb{E}\left\Vert g\left(A\right)^{*}u\right\Vert ^{2}=\left\Vert P^{1/2}g\left(S\right)^{*}Wu\right\Vert ^{2}\label{eq:c4}
\end{equation}
for all $g\in\mathcal{A}_{d}$. 

\textit{Define $\Phi$ by radial ms-SOT limits}. Fix $\phi\in F^{\infty}_{d}$.
For any $u\in H$, apply \eqref{eq:c4} to $g=\phi_{r}-\phi_{s}$:
\[
\mathbb{E}\left\Vert \left(\phi_{r}\left(A\right)-\phi_{s}\left(A\right)\right)^{*}u\right\Vert ^{2}=\left\Vert P^{1/2}\left(\phi_{r}\left(S\right)-\phi_{s}\left(S\right)\right)^{*}Wu\right\Vert ^{2}.
\]
Since $\left\{ \phi_{r}\left(S\right)\right\} {}_{0<r<1}$ is uniformly
bounded ($\left\Vert \phi_{r}\left(S\right)\right\Vert \le\left\Vert \phi\left(L\right)\right\Vert $)
and $\phi_{r}\left(S\right)\rightarrow\phi\left(S\right)$ in SOT
as $r\uparrow1$, the right-hand side $\rightarrow0$ for each fixed
$u$. Thus $\left\{ \phi_{r}\left(A\right)^{*}u\right\} _{r}$ is
Cauchy in $L^{2}(\Omega,H)$. Define 
\[
\left(\Phi\left(\phi\right)\right)^{*}u\coloneqq\lim_{r\uparrow1}\left(\phi_{r}\left(A\right)\right)^{*}u.
\]
 Linearity of $\Phi$ follows from $\left(a\phi+b\psi\right)_{r}=a\phi_{r}+b\psi_{r}$
and linearity of $L^{2}$ limits. 

\textit{Contractivity}. Using \eqref{eq:c4} with $g=\phi_{r}$ and
letting $r\uparrow1$, 
\begin{align*}
\mathbb{E}\left\Vert \left(\Phi\left(\phi\right)\right)^{*}u\right\Vert ^{2} & =\lim_{r\uparrow1}\mathbb{E}\left\Vert \phi_{r}\left(A\right)^{*}u\right\Vert ^{2}\\
 & =\lim_{r\uparrow1}\left\Vert P^{1/2}\phi_{r}\left(S\right)^{*}Wu\right\Vert ^{2}\\
 & =\big\| P^{1/2}\left(\phi\left(S\right)\right)^{*}Wu\big\|^{2}\\
 & \leq\left\Vert \phi\left(S\right)\right\Vert ^{2}\left\Vert Wu\right\Vert ^{2}\le\left\Vert \phi\left(L\right)\right\Vert ^{2}\left\Vert u\right\Vert ^{2}.
\end{align*}
Hence
\[
\mathbb{E}\left[\Phi\left(\phi\right)\left(\Phi\left(\phi\right)\right)^{*}\right]\le\left\Vert \phi\left(L\right)\right\Vert ^{2}I_{H}.
\]

\textit{Compatibility}. If $\phi\in\mathcal{A}_{d}$, then $\phi_{r}\left(S\right)\rightarrow\phi\left(S\right)$
in norm. By \eqref{eq:c4}, 
\[
\mathbb{E}\left\Vert \left(\phi_{r}\left(A\right)-\phi\left(A\right)\right)^{*}u\right\Vert ^{2}\rightarrow0.
\]
Thus, $\Phi\left(\phi\right)=\phi\left(A\right)$. 

\textit{Uniqueness}. If $\Psi$ is another linear map with the same
ms-SOT radial limit property, then for all $u\in H$, 
\[
\left(\Psi\left(\phi\right)\right)^{*}u=\lim_{r\uparrow1}\left(\phi_{r}\left(A\right)\right)^{*}u=\left(\Phi\left(\phi\right)\right)^{*}u.
\]
Hence $\Psi\left(\phi\right)=\Phi\left(\phi\right)$ in mean-square
sense. 
\end{proof}

\section{Multi-Level Moment Defects and Nested Dilations}\label{sec:4}

In this section, we extend the one-step defect analysis of \prettyref{thm:1}
to an arbitrary finite or infinite hierarchy of higher-order defects.
We show that if a moment kernel admits all iterated defect kernels,
then a single Cuntz dilation simultaneously realizes every defect
kernel through a nested chain of projections. This produces a canonical
decomposition of the moment structure.

Let $d\in\mathbb{N}$. For a positive definite kernel $K:\mathbb{F}^{+}_{d}\times\mathbb{F}^{+}_{d}\to\mathcal{L}\left(H\right)$,
define the free shift by 
\[
\left(\Sigma K\right)\left(\alpha,\beta\right):=\sum^{d}_{i=1}K\left(\alpha i,\beta i\right).
\]
Define the hierarchy of defect kernels recursively by 
\begin{align*}
K^{(0)} & :=K,\\
K^{(n+1)} & :=\left(I-\Sigma\right)K^{(n)},\qquad n\ge0.
\end{align*}
Explicitly, we have $K^{(1)}=K-\Sigma K$, $K^{(2)}=K-2\Sigma K+\Sigma^{2}K$,
and so forth. We are interested in the case where every defect kernel
is positive definite.

The following result shows that the positivity of the defect chain
$\{K^{(n)}\}$ is equivalent to the existence of a single Cuntz family
$S$ and a single isometric embedding $W:H\to\mathcal{K}$, together
with a nested sequence of projections 
\[
P_{0}\ge P_{1}\ge\cdots
\]
where each projection captures the defect structure at a fixed level.
\begin{thm}[Nested defects]
\label{thm:4-1} Let $M\in\mathbb{N}\cup\left\{ \infty\right\} $.
The following are equivalent:
\begin{enumerate}
\item[(A)] For every $n=1,\dots,M$, the iterated defect kernel 
\[
K^{(n)}=\left(I-\Sigma\right)^{n}K
\]
is positive definite on $\mathbb{F}^{+}_{d}\times\mathbb{F}^{+}_{d}$. 
\item[(B)] There exist a Hilbert space $\mathcal{K}$, a Cuntz family $S=\left(S_{1},\dots,S_{d}\right)$
on $\mathcal{K}$, an isometry $W:H\to\mathcal{K}$, and projections
\[
P_{0}\ge P_{1}\ge\cdots\ge P_{M}\ge0,
\]
such that for every $0\le n\le M$, 
\[
K^{(n)}\left(\alpha,\beta\right)=W^{*}S^{\alpha}P_{n}\left(S^{\beta}\right)^{*}W,\qquad\alpha,\beta\in\mathbb{F}^{+}_{d}.
\]
Each projection satisfies the defect recursion on the cyclic subspace
\begin{equation}
\mathcal{M}:=\overline{span}\left\{ \left(S^{\gamma}\right)^{*}Wu:\gamma\in\mathbb{F}^{+}_{d},u\in H\right\} ,\label{eq:4-1}
\end{equation}
namely, 
\[
\sum^{d}_{i=1}S_{i}P_{n}S^{*}_{i}\le P_{n}\quad\text{on }\mathcal{M},
\]
and 
\[
P_{n+1}=P_{n}-\sum^{d}_{i=1}S_{i}P_{n}S^{*}_{i}.
\]
\end{enumerate}
Moreover, the following properties hold: 
\begin{enumerate}
\item[(i)] For every finite $N\le M$, 
\[
K=\Sigma^{N}K+\sum^{N-1}_{n=0}\Sigma^{n}K^{(n+1)}.
\]
\item[(ii)]  If $\Sigma^{N}K\to0$ in the quadratic-form sense on finite index
sets, then 
\[
K=\sum^{\infty}_{n=0}\Sigma^{n}K^{(n+1)}.
\]
\item[(iii)] In terms of the defect projections, 
\[
P_{0}=\sum^{\infty}_{n=0}T^{n}\left(P_{n+1}\right)\quad\text{on }\mathcal{M},\qquad T\left(X\right):=\sum^{d}_{i=1}S_{i}XS^{*}_{i}.
\]
\item[(iv)]  If $\mathcal{K}$ is the closed span of 
\[
\left\{ S^{\alpha}P_{n}\left(S^{\beta}\right)^{*}Wu:\alpha,\beta\in\mathbb{F}^{+}_{d},0\le n\le M,u\in H\right\} ,
\]
then the dilation is minimal and unique up to a unitary intertwining
the Cuntz families. 
\end{enumerate}
\end{thm}

\begin{proof}
(A)$\Rightarrow$(B). We proceed by induction on $n$. The base case
($n=0,1$) is exactly \prettyref{thm:1}: We have 
\[
K\left(\alpha,\beta\right)=W^{*}S^{\alpha}P_{0}\left(S^{\beta}\right)^{*}W,
\]
and $\sum^{d}_{i=1}S_{i}P_{0}S^{*}_{i}\le P_{0}$ on the cyclic subspace
$\mathcal{M}$; see \eqref{eq:4-1}. Define 
\[
P_{1}:=P_{0}-\sum^{d}_{i=1}S_{i}P_{0}S^{*}_{i}.
\]
Then for every $\alpha,\beta$, 
\[
K^{(1)}\left(\alpha,\beta\right)=K\left(\alpha,\beta\right)-\Sigma K\left(\alpha,\beta\right)=W^{*}S^{\alpha}P_{1}\left(S^{\beta}\right)^{*}W.
\]
The positivity of $K^{(1)}$ implies $P_{1}\ge0$ on $\mathcal{M}$. 

For the inductive step, assume for some $1\le n<M$ that 
\[
K^{(n)}\left(\alpha,\beta\right)=W^{*}S^{\alpha}P_{n}\left(S^{\beta}\right)^{*}W
\]
with $P_{n}\ge0$ on $\mathcal{M}$ and $\sum^{d}_{i=1}S_{i}P_{n}S^{*}_{i}\le P_{n}$
on $\mathcal{M}$. We compute: 
\begin{align*}
\Sigma K^{(n)}\left(\alpha,\beta\right) & =\sum^{d}_{i=1}W^{*}S^{\alpha i}P_{n}\left(S^{\beta i}\right)^{*}W\\
 & =W^{*}S^{\alpha}\left(\sum^{d}_{i=1}S_{i}P_{n}S^{*}_{i}\right)\left(S^{\beta}\right)^{*}W.
\end{align*}
Thus, 
\[
K^{(n+1)}\left(\alpha,\beta\right)=K^{(n)}\left(\alpha,\beta\right)-\Sigma K^{(n)}\left(\alpha,\beta\right)=W^{*}S^{\alpha}P_{n+1}\left(S^{\beta}\right)^{*}W,
\]
where 
\[
P_{n+1}:=P_{n}-\sum^{d}_{i=1}S_{i}P_{n}S^{*}_{i}.
\]
The positivity of $K^{(n+1)}$ implies $P_{n+1}\ge0$ on $\mathcal{M}$.
This completes the induction.

(B)$\Rightarrow$(A). Assume the structure in (B). For a finite family
$\left\{ \left(\alpha_{j},u_{j}\right)\right\} $, 
\[
\sum_{j,k}\left\langle u_{j},K^{(n)}\left(\alpha_{j},\alpha_{k}\right)u_{k}\right\rangle =\left\langle \sum_{j}\left(S^{\alpha_{j}}\right)^{*}Wu_{j},P_{n}\sum_{k}\left(S^{\alpha_{k}}\right)^{*}Wu_{k}\right\rangle \ge0.
\]
Thus $K^{(n)}$ is positive definite. This proves the equivalence.

\textit{Expansions and Uniqueness.} Since $K^{(n)}=\Sigma K^{(n)}+K^{(n+1)}$,
iterating yields $K=\Sigma^{N}K+\sum^{N-1}_{n=0}\Sigma^{n}K^{(n+1)}$
for any $N\le M$. If $\Sigma^{N}K\to0$ in the quadratic form on
finite index sets, taking limits yields the infinite expansion. Under
the dilation representation, 
\[
\Sigma^{n}K^{(n+1)}\left(\alpha,\beta\right)=W^{*}S^{\alpha}T^{n}\left(P_{n+1}\right)\left(S^{\beta}\right)^{*}W,
\]
so $P_{0}=\sum^{\infty}_{n=0}T^{n}\left(P_{n+1}\right)$ on $\mathcal{M}$.
Minimality follows by construction from the closure of the span of
$\{S^{\alpha}P_{n}\left(S^{\beta}\right)^{*}Wu\}$. Uniqueness up
to a unitary follows from the uniqueness of the Kolmogorov factorization
and minimal Cuntz dilation. 
\end{proof}
The theorem produces a canonical hierarchy inside the Cuntz dilation,
with projections $P_{0}\ge P_{1}\ge\cdots$ encoding ``moment defects''
at all depths. When $\Sigma^{N}K\to0$, the moment structure decomposes
into shifted positive layers, analogous to discrete completely monotone
sequences. 
\begin{example}
We consider a concrete scalar example. Fix a probability measure $\mu$
on $\left[0,1\right]$. Define the moment kernel for $m,n\in\mathbb{N}$
by 
\[
K\left(m,n\right):=\int^{1}_{0}x^{m+n}d\mu\left(x\right).
\]

For $d=1$, the shift operation corresponds to $\left(\Sigma K\right)\left(m,n\right)=K\left(m+1,n+1\right)$.
The first defect kernel is 
\begin{align*}
K^{(1)}\left(m,n\right) & =K\left(m,n\right)-K\left(m+1,n+1\right)\\
 & =\int^{1}_{0}x^{m+n}\left(1-x^{2}\right)d\mu\left(x\right).
\end{align*}
Since $1-x^{2}\ge0$ on $\left[0,1\right]$, the kernel $K^{(1)}$
is positive definite. By induction, we verify the higher-order defects.
For a fixed $x$, let $k_{x}\left(m,n\right)=x^{m+n}$. Observe that
$\left(I-\Sigma\right)k_{x}=\left(1-x^{2}\right)k_{x}$. Iterating
this applies the weight function $\left(1-x^{2}\right)$ repeatedly.
Thus, 
\[
K^{(n)}\left(m,n\right)=\int^{1}_{0}x^{m+n}\left(1-x^{2}\right)^{n}d\mu\left(x\right).
\]
Since $\left(1-x^{2}\right)^{n}\ge0$ on the support of $\mu$, every
iterated defect kernel $K^{(n)}$ is positive definite. This confirms
that Condition (A) of \prettyref{thm:4-1} is satisfied for every
probability measure on $\left[0,1\right]$.

The asymptotic shift corresponds to 
\[
\left(\Sigma^{N}K\right)\left(m,n\right)=\int^{1}_{0}x^{m+n}\left(x^{2}\right)^{N}d\mu\left(x\right).
\]
As $N\to\infty$, the term $x^{2N}$ converges to the indicator function
of the set $\left\{ 1\right\} $. By the monotone convergence theorem,
\[
\lim_{N\to\infty}\left(\Sigma^{N}K\right)\left(m,n\right)=\mu\left(\left\{ 1\right\} \right).
\]
If $\mu$ has no atom at $1$, then $\Sigma^{N}K\to0$, and the kernel
admits the infinite telescoping decomposition 
\[
K=\sum^{\infty}_{n=0}\Sigma^{n}K^{(n+1)}.
\]
In this case, the moment kernel is entirely captured by the defects
of the random contraction.
\end{example}

\section{Wold-type Decomposition}\label{sec:5}

In the classical one-variable setting, the Wold decomposition for
an isometry separates the dynamics into a ``pure shift'' part and
a ``unitary'' part by looking at the asymptotics of $\left(V^{n}\right)$
or $\left(V^{*n}\right)$. Multivariable and noncommutative versions
of this picture for row isometries, free semigroup representations,
and completely positive maps all share the same basic feature: there
is a canonical invariant subspace on which nothing decays, and a complementary
pure part that is driven to zero under iteration. In the kernel setting,
there is no operator playing the role of a single isometry; instead,
the free shift $\Sigma$ acts at the level of the moment kernel $K$.
The contractivity condition $\left(K^{(1)}=K-\Sigma K\ge0\right)$
implies that the iterates 
\begin{equation}
K\ge\Sigma K\ge\Sigma^{2}K\ge\cdots\label{eq:5-1}
\end{equation}
form a decreasing chain of positive definite kernels. It is therefore
natural to ask whether one can recover a Wold-type splitting directly
from this chain: is there a maximal part of $K$ that is invariant
under $\Sigma$, and a complementary part that is asymptotically negligible
under iteration? The purpose of this section is to show that the answer
is yes. We prove that every contractive random moment kernel admits
a canonical decomposition 
\[
K=K_{\mathrm{pure}}+K_{\mathrm{unitary}},
\]
where $K_{\mathrm{unitary}}$ is shift-invariant $\left(\Sigma K_{\mathrm{unitary}}=K_{\mathrm{unitary}}\right)$
, $K_{\mathrm{pure}}$ is asymptotically pure $\left(\Sigma^{N}K_{\mathrm{pure}}\to0\text{ in the point-weak sense}\right)$,
and this splitting is unique and maximal in an appropriate order-theoretic
sense. We then show that this kernel-level Wold decomposition can
be realized inside the minimal Cuntz dilation of $K$: the unitary
part is obtained by compressing an asymptotic projection for the dilation,
and the pure part corresponds to its orthogonal complement. In this
way, the familiar dichotomy between pure and unitary dynamics is transported
to the level of random moment kernels.

Throughout this section we assume that $K:\mathbb{F}^{+}_{d}\times\mathbb{F}^{+}_{d}\to\mathcal{L}\left(H\right)$
is a moment kernel satisfying the basic contractivity condition 
\[
K^{\left(1\right)}:=K-K_{\Sigma}\ge0,
\]
and we continue to write $\left(\Sigma K\right)\left(\alpha,\beta\right)=\sum^{d}_{i=1}K\left(\alpha i,\beta i\right)$. 
\begin{lem}
$\Sigma$ is order-preservaing, i.e., if $K\leq L$, then $\Sigma K\leq\Sigma L$. 
\end{lem}

\begin{proof}
If $K$ is p.d., then for any finite family $\left\{ \left(\alpha_{j},u_{j}\right)\right\} $,
\[
\sum_{j,k}\left\langle u_{j},\left(\Sigma K\right)\left(\alpha_{j},\alpha_{k}\right)u_{k}\right\rangle =\sum^{d}_{i=1}\sum_{j,k}\left\langle u_{j},K\left(\alpha_{j}i,\alpha_{k}i\right)u_{k}\right\rangle \geq0
\]
so $\Sigma K$ is p.d. If $K\leq L$, then $L-K\geq0$, hence $\Sigma\left(L-K\right)\geq0$,
i.e., $\Sigma K\leq\Sigma L$. 
\end{proof}
Iteration and order-preservation of $\Sigma$ then give a decreasing
chain of positive definite kernels as in \eqref{eq:5-1}.

Next, we show that this chain admits a canonical asymptotic limit,
and that this limit determines a unique decomposition of $K$ into
a ``pure'' part and a ``unitary'' part. 
\begin{thm}
\label{thm:5-1}Let $K$ be a positive definite moment kernel with
$K^{(1)}\ge0$. Then there exist unique positive definite kernels
$K_{\text{pure}}$ and $K_{\text{unitary}}$ such that 
\[
K=K_{\text{pure}}+K_{\text{unitary}},
\]
and the following properties hold.
\begin{enumerate}
\item $K_{\text{unitary}}$ is invariant under the free shift: 
\[
\Sigma K_{\text{unitary}}=K_{\text{unitary}}.
\]
\item $K_{\text{pure}}$ is stable under iteration of the shift, in the
sense that 
\[
\Sigma^{N}K_{\text{pure}}\longrightarrow0
\]
in the point-weak topology.
\end{enumerate}
Moreover, 
\[
K_{\text{unitary}}\left(\alpha,\beta\right)=\lim_{N\to\infty}\left(\Sigma^{N}K\right)\left(\alpha,\beta\right),\qquad\alpha,\beta\in\mathbb{F}^{+}_{d}.
\]

\end{thm}

\begin{proof}
Since $\Sigma K\le K$, the iterates satisfy $K\ge\Sigma K\ge\Sigma^{2}K\ge\cdots$.
Fix a finite family $\left\{ \left(\alpha_{j},u_{j}\right)\right\} \subset\mathbb{F}^{+}_{d}\times H$
and set 
\[
Q_{N}:=\sum_{j,k}\left\langle u_{j},\left(\Sigma^{N}K\right)\left(\alpha_{j},\alpha_{k}\right)u_{k}\right\rangle .
\]
Because the kernels form a decreasing sequence in the positive definite
order, the real sequence $Q_{N}$ is decreasing and bounded below
by $0$. Hence it converges. Define $K_{\infty}$ by 
\[
\sum_{j,k}\left\langle u_{j},K_{\infty}\left(\alpha_{j},\alpha_{k}\right)u_{k}\right\rangle =\lim_{N\to\infty}Q_{N}.
\]
By standard polarization, $K_{\infty}$ is a well-defined positive
definite kernel.

To verify invariance, note that 
\begin{align*}
\left(\Sigma K_{\infty}\right)\left(\alpha,\beta\right) & =\sum^{d}_{i=1}K_{\infty}\left(\alpha i,\beta i\right)=\sum^{d}_{i=1}\lim_{N\to\infty}\left(\Sigma^{N}K\right)\left(\alpha i,\beta i\right)\\
 & =\lim_{N\to\infty}\left(\Sigma^{N+1}K\right)\left(\alpha,\beta\right)=K_{\infty}\left(\alpha,\beta\right),
\end{align*}
so $\Sigma K_{\infty}=K_{\infty}$. Set $K_{\text{unitary}}:=K_{\infty}$
and $K_{\text{pure}}:=K-K_{\infty}$. Since $K_{\infty}\le K$, the
kernel $K_{\text{pure}}$ is positive definite.

Finally, using the identity $\Sigma^{N}K_{\infty}=K_{\infty}$, 
\[
\left(\Sigma^{N}K_{\text{pure}}\right)\left(\alpha,\beta\right)=\left(\Sigma^{N}K\right)\left(\alpha,\beta\right)-K_{\infty}\left(\alpha,\beta\right),
\]
and the right-hand side converges to $0$ because $\Sigma^{N}K\to K_{\infty}$
pointwise in the quadratic forms. Uniqueness follows from the fact
that any other decomposition with the same properties agrees in the
limit under $\Sigma^{N}$ and therefore coincides with the one above. 
\end{proof}
\begin{cor}
With notation as in \prettyref{thm:5-1}, the kernel $K_{\text{unitary}}$
is the largest positive definite kernel $L$ dominated by $K$ and
fixed by the shift: 
\[
L\le K,\quad\Sigma L=L\Longrightarrow L\le K_{\text{unitary}}.
\]
Moreover, $K_{\text{pure}}$ contains no nontrivial invariant part:
if $L\le K_{\text{pure}}$ and $\Sigma L=L$, then $L=0$.
\end{cor}

\begin{proof}
Suppose $L\ge0$, $L\le K$, and $\Sigma L=L$. By order preservation
of $\Sigma$, 
\[
L=\Sigma^{N}L\le\Sigma^{N}K
\]
for every $N$. Passing to the limit of quadratic forms gives $L\le K_{\text{unitary}}$,
proving maximality.

For the second assertion, if $L\le K_{\text{pure}}$ and $\Sigma L=L$,
then $L\le K$ and the preceding argument gives $L\le K_{\text{unitary}}$.
Since $K_{\text{pure}}$ and $K_{\text{unitary}}$ have disjoint support
in the sense that 
\[
K_{\text{pure}}=K-K_{\text{unitary}},
\]
one must have $L=0$. 
\end{proof}
\begin{cor}
\label{cor:5-3}Let $\left(\mathcal{K},S,W,P\right)$ be the minimal
Cuntz dilation of $K$ constructed in \prettyref{thm:4-1}. Define
the completely positive map 
\[
T\left(X\right):=\sum^{d}_{i=1}S_{i}XS^{*}_{i}.
\]
Then for every $N\ge0$, 
\[
\left(\Sigma^{N}K\right)\left(\alpha,\beta\right)=W^{*}S^{\alpha}T^{N}\left(P\right)\left(S^{\beta}\right)^{*}W.
\]
The limit defining $K_{\text{unitary}}$ in \prettyref{thm:5-1} is
realized through the strong-operator limit of the sequence $T^{N}\left(P\right)$
on the cyclic subspace $\mathcal{M}$ in \eqref{eq:4-1}. Defining
\[
\left\langle y,P_{\infty}y\right\rangle :=\lim_{N\to\infty}\left\langle y,T^{N}\left(P\right)y\right\rangle ,\qquad y\in\mathcal{M},
\]
gives a well-defined positive contraction $P_{\infty}$ on $\mathcal{M}$
satisfying 
\begin{align*}
K_{\text{unitary}}\left(\alpha,\beta\right) & =W^{*}S^{\alpha}P_{\infty}\left(S^{\beta}\right)^{*}W,\\
K_{\text{pure}}\left(\alpha,\beta\right) & =W^{*}S^{\alpha}\left(P-P_{\infty}\right)\left(S^{\beta}\right)^{*}W.
\end{align*}
\end{cor}

\begin{proof}
The identity for $\Sigma^{N}K$ follows by induction from the formula
in \prettyref{thm:4-1}: 
\[
\left(\Sigma K\right)\left(\alpha,\beta\right)=W^{*}S^{\alpha}T\left(P\right)\left(S^{\beta}\right)^{*}W.
\]
Iterating yields 
\[
\left(\Sigma^{N}K\right)\left(\alpha,\beta\right)=W^{*}S^{\alpha}T^{N}\left(P\right)\left(S^{\beta}\right)^{*}W.
\]
For $y=\sum_{j}\left(S^{\alpha_{j}}\right)^{*}Wu_{j}\in\mathcal{M}$,
a standard computation using the Cuntz relations shows that 
\[
\left\langle y,T^{N}\left(P\right)y\right\rangle =\sum_{j,k}\left\langle u_{j},\left(\Sigma^{N}K\right)\left(\alpha_{j},\alpha_{k}\right)u_{k}\right\rangle .
\]
As $N\to\infty$, this converges to the quadratic form of $K_{\text{unitary}}$,
proving that 
\[
K_{\text{unitary}}\left(\alpha,\beta\right)=W^{*}S^{\alpha}P_{\infty}\left(S^{\beta}\right)^{*}W.
\]
The formula for $K_{\text{pure}}$ follows from $K_{\text{pure}}=K-K_{\text{unitary}}$. 
\end{proof}

\section{A curvature-type Invariant}\label{sec:6}

A curvature-type invariant sits naturally on top of the Wold picture
from \prettyref{sec:5}. Once we know that every contractive random
moment kernel $K$ admits a canonical splitting into a ``pure''
part that eventually dissipates under the free shift and a ``unitary''
part that survives unchanged, it is reasonable to ask for a single
numerical quantity that measures how large the non-dissipating component
really is. In the finite-dimensional, normalized setting, such a number
ought to live between $0$ and $\dim H$, vanish exactly when the
kernel is purely dissipative, and be recoverable from the asymptotic
behaviour of the iterates $\Sigma^{N}K$. This kind of scalar asymptotic
invariant has a clear analogue in multivariable operator theory. Arveson's
curvature for Hilbert modules and row contractions, and its subsequent
refinements in the noncommutative setting by Popescu and others, extract
from a dilation or a module a number that balances an initial ``dimension''
against the total accumulated defect; see, for instance, Arveson's
work on submodules of the $d$-shift and curvature invariants for
Hilbert modules \cite{MR1712636,MR1758582}, and Popescu's curvature
invariants for noncommutative multivariable operator theory \cite{MR1822685,MR3345180},
as well as related noncommutative variants in work of Kribs \cite{MR1857801}.
In those contexts, curvature detects the unitary part of a tuple and
plays a role in classification and rigidity results. 

The goal of this section is to introduce an analogous scalar in the
setting of random moment kernels. For a normalized kernel $K$ with
$K^{(1)}\ge0$, the Wold decomposition from \prettyref{sec:5} produces
a canonical invariant kernel $K_{\mathrm{unitary}}$ and a pure kernel
$K_{\mathrm{pure}}$. We will show that in the finite-dimensional
case the asymptotic ``mass'' of $K$ under the free shift, 
\[
\mathrm{tr}\left(\left(\Sigma^{N}K\right)(\emptyset,\emptyset)\right),
\]
converges and defines a curvature-type invariant $\kappa(K)$ which
coincides with 
\[
\mathrm{tr}\left(K_{\mathrm{unitary}}(\emptyset,\emptyset)\right),
\]
admits a defect-summation formula, and vanishes precisely when $K$
is purely dissipative. In this way, $\kappa(K)$ plays for random
moment kernels the same structural role that curvature plays for classical
row contractions and Hilbert modules.

Throughout this section $H$ is a finite-dimensional Hilbert space
and 
\[
K:\mathbb{F}^{+}_{d}\times\mathbb{F}^{+}_{d}\to\mathcal{L}\left(H\right)
\]
is a moment kernel arising from a random $d$-tuple $\left(A_{1},\dots,A_{d}\right)$,
normalized by 
\[
K\left(\emptyset,\emptyset\right)=I_{H}.
\]
We assume the first defect positivity 
\[
K^{(1)}:=K-\Sigma K\ge0,
\]
where the shift $\Sigma$ is given by 
\[
\left(\Sigma K\right)\left(\alpha,\beta\right)=\sum^{d}_{i=1}K\left(\alpha i,\beta i\right).
\]
By \prettyref{thm:5-1}, this assumption implies the existence of
a canonical Wold decomposition 
\[
K=K_{\text{pure}}+K_{\text{unitary}},
\]
where $K_{\text{unitary}}$ is shift-invariant ($\Sigma K_{\text{unitary}}=K_{\text{unitary}}$)
while $K_{\text{pure}}$ dissipates under iteration ($\Sigma^{N}K_{\text{pure}}\to0$
point-weakly). The operator-theoretic realization of this decomposition
inside the minimal Cuntz dilation was established in \prettyref{cor:5-3}.

In this section we introduce a numerical invariant associated with
the unitary part and obtain several equivalent formulas for it. We
also show that this invariant satisfies a noncommutative defect-summation
identity.

For each $N\ge0$, the kernel $\Sigma^{N}K$ satisfies 
\[
0\le\left(\Sigma^{N+1}K\right)\left(\emptyset,\emptyset\right)\le\left(\Sigma^{N}K\right)\left(\emptyset,\emptyset\right)\le I_{H}.
\]
Thus the sequence of positive operators 
\[
\left(\Sigma^{N}K\right)\left(\emptyset,\emptyset\right)
\]
is monotonically decreasing and bounded below by $0$, hence convergent
in the operator norm (and in every operator topology) in finite dimension.
\begin{defn}
\label{def:6-1}The curvature of the moment kernel $K$ is the scalar
\[
\kappa\left(K\right):=\lim_{N\to\infty}\text{tr}\left(\left(\Sigma^{N}K\right)\left(\emptyset,\emptyset\right)\right).
\]
Since each term is bounded by $\text{tr}\left(I_{H}\right)=\dim\left(H\right)$,
the limit exists and satisfies 
\[
0\le\kappa\left(K\right)\le\dim\left(H\right).
\]
This quantity measures the non-dissipating part of the random tuple
under repeated application of the free shift. The connection to the
Wold decomposition will be made precise in the following theorem. 
\end{defn}

We now give five equivalent descriptions of $\kappa\left(K\right)$
involving (i) the unitary kernel, (ii) a variational supremum, and
(iii) the asymptotic projection inside the dilation.
\begin{thm}
\label{thm:6-2}Let $K=K_{\text{pure}}+K_{\text{unitary}}$ be the
Wold decomposition from \prettyref{thm:5-1}, and let $\left(\mathcal{K},S,W,P\right)$
be the minimal Cuntz dilation of $K$ (\prettyref{thm:4-1}), with
asymptotic projection $P_{\infty}$ constructed in \prettyref{cor:5-3}.
Then the following quantities are equal: 
\[
\begin{aligned}\kappa\left(K\right) & =\lim_{N\to\infty}\text{tr}\left(\left(\Sigma^{N}K\right)\left(\emptyset,\emptyset\right)\right)\\
 & =\text{tr}\left(K_{\text{unitary}}\left(\emptyset,\emptyset\right)\right)\\
 & =\sup\left\{ \text{tr}\left(L\left(\emptyset,\emptyset\right)\right):0\le L\le K,\ \Sigma L=L\right\} \\
 & =\text{tr}\left(W^{*}P_{\infty}W\right)\\
 & =\text{tr}\left(P_{\infty}WW^{*}\right).
\end{aligned}
\]
Furthermore: 
\[
\kappa\left(K\right)=0\Longleftrightarrow K_{\text{unitary}}=0\Longleftrightarrow\Sigma^{N}K\to0,
\]
and 
\[
\kappa\left(K\right)=\dim\left(H\right)\quad\Longleftrightarrow\quad K_{\text{pure}}=0.
\]
\end{thm}

\begin{proof}
\textit{Step 1. Equality with the unitary part.} By \prettyref{thm:5-1},
\[
\left(\Sigma^{N}K\right)\left(\alpha,\beta\right)\xrightarrow[N\to\infty]{}K_{\text{unitary}}\left(\alpha,\beta\right)
\]
in the point-weak topology. Since in finite dimension convergence
of matrices in the weak topology implies convergence in norm, we have
\[
\lim_{N\to\infty}\left(\Sigma^{N}K\right)\left(\emptyset,\emptyset\right)=K_{\text{unitary}}\left(\emptyset,\emptyset\right)
\]
as operators. Taking traces gives $\kappa\left(K\right)=\text{tr}\left(K_{\text{unitary}}\left(\emptyset,\emptyset\right)\right)$.

\textit{Step 2. Variational characterization.} Let $L$ be a positive
definite kernel with $0\le L\le K$ and $\Sigma L=L$. Applying $\Sigma^{N}$
to the inequality gives 
\[
L=\Sigma^{N}L\le\Sigma^{N}K.
\]
Taking the $\left(\emptyset,\emptyset\right)$-entry and the trace
yields 
\[
\text{tr}\left(L\left(\emptyset,\emptyset\right)\right)\le\text{tr}\left(\left(\Sigma^{N}K\right)\left(\emptyset,\emptyset\right)\right).
\]
Letting $N\to\infty$ and using Step 1 gives $\text{tr}\left(L\left(\emptyset,\emptyset\right)\right)\le\text{tr}\left(K_{\text{unitary}}\left(\emptyset,\emptyset\right)\right)$.
Taking the supremum over all such $L$ gives the inequality in one
direction. The reverse inequality holds by choosing $L=K_{\text{unitary}}$,
proving the variational formula.

\textit{Step 3. Dilation formulas.} \prettyref{cor:5-3} shows that
\[
K_{\text{unitary}}\left(\alpha,\beta\right)=W^{*}S^{\alpha}P_{\infty}\left(S^{\beta}\right)^{*}W.
\]
Thus 
\[
K_{\text{unitary}}\left(\emptyset,\emptyset\right)=W^{*}P_{\infty}W.
\]
Taking traces gives the fourth formula. To obtain the fifth, note
that $WW^{*}$ is the projection onto $\text{ran}\left(W\right)\subseteq\mathcal{K}$,
which has finite rank $\dim\left(H\right)$. Thus the operator $P_{\infty}WW^{*}$
is finite-rank and trace-class, and the identity $\text{tr}\left(W^{*}P_{\infty}W\right)=\text{tr}\left(P_{\infty}WW^{*}\right)$.

\textit{Step 4. Dichotomy properties.} If $\kappa\left(K\right)=0$,
then $K_{\text{unitary}}\left(\emptyset,\emptyset\right)=0$. Since
$K_{\text{unitary}}\ge0$ and shift-invariant, the identity $K_{\text{unitary}}\left(\alpha,\alpha\right)=\sum^{d}_{i=1}K_{\text{unitary}}\left(\alpha i,\alpha i\right)$
and positivity imply $K_{\text{unitary}}\left(\alpha,\alpha\right)=0$
for all $\alpha$, and the polarization identity then forces $K_{\text{unitary}}=0$.
Conversely, if $K_{\text{unitary}}=0$, then $\kappa\left(K\right)=0$
by Step 1. This implies $\Sigma^{N}K=\Sigma^{N}K_{\text{pure}}\to0$.

If $\kappa\left(K\right)=\dim\left(H\right)=\text{tr}\left(I_{H}\right)$,
then $\text{tr}\left(K_{\text{unitary}}\left(\emptyset,\emptyset\right)\right)=\text{tr}\left(I_{H}\right)$,
so positivity implies $K_{\text{unitary}}\left(\emptyset,\emptyset\right)=I_{H}$.
Since $K\left(\emptyset,\emptyset\right)=I_{H}$, this forces $K_{\text{pure}}\left(\emptyset,\emptyset\right)=0$.
Monotonicity implies each term vanishes identically, and again polarization
gives $K_{\text{pure}}=0$. The converse is immediate. 
\end{proof}

\subsection*{Defect summation identity}

We now express $\kappa\left(K\right)$ in terms of the cumulative
sum of the first defect kernel. This formula is a direct analogue
of classical curvature identities in single-variable dilation theory.
\begin{thm}
\label{thm:6-3}For every integer $N\ge1$, 
\[
K-\Sigma^{N}K=\sum^{N-1}_{n=0}\Sigma^{n}K^{(1)}.
\]
Taking traces at $\left(\emptyset,\emptyset\right)$ and letting $N\to\infty$
yields the identity 
\[
\kappa\left(K\right)=\text{tr}\left(I_{H}\right)-\sum^{\infty}_{n=0}\text{tr}\left(\left(\Sigma^{n}K^{(1)}\right)\left(\emptyset,\emptyset\right)\right).
\]
The series on the right converges in $\left[0,\text{tr}\left(I_{H}\right)\right]$. 
\end{thm}

\begin{proof}
By definition, $K^{(1)}=K-\Sigma K$. Applying $\Sigma^{n}$ gives
\[
\Sigma^{n}K^{(1)}=\Sigma^{n}K-\Sigma^{n+1}K\qquad\left(n\ge0\right).
\]
Summing these identities for $n=0,\dots,N-1$ gives telescoping: 
\begin{align*}
\sum^{N-1}_{n=0}\Sigma^{n}K^{(1)} & =\left(K-\Sigma K\right)+\left(\Sigma K-\Sigma^{2}K\right)+\cdots+\left(\Sigma^{N-1}K-\Sigma^{N}K\right)\\
 & =K-\Sigma^{N}K.
\end{align*}
Now take $\left(\emptyset,\emptyset\right)$-entries and apply the
trace. Since all operators involved are positive semidefinite and
dominated by $I_{H}$, each term has finite trace. Thus 
\[
\text{tr}\left(K\left(\emptyset,\emptyset\right)\right)-\text{tr}\left(\left(\Sigma^{N}K\right)\left(\emptyset,\emptyset\right)\right)=\sum^{N-1}_{n=0}\text{tr}\left(\left(\Sigma^{n}K^{(1)}\right)\left(\emptyset,\emptyset\right)\right).
\]
As $N\to\infty$, the left-hand side converges to $\text{tr}\left(I_{H}\right)-\kappa\left(K\right)$,
by \prettyref{def:6-1} and \prettyref{thm:6-2}. All terms on the
right are nonnegative, so the series converges by monotone convergence.
Rewriting yields the stated identity. 
\end{proof}

\subsection*{Consequences and examples}

We record two immediate implications of the defect summation identity.
\begin{cor}[Quantitative estimate under defect decay ]
Suppose there exist constants $C>0$ and $0<\rho<1$ such that 
\[
\text{tr}\left(\left(\Sigma^{n}K^{(1)}\right)\left(\emptyset,\emptyset\right)\right)\le C\rho^{n}\qquad\left(n\ge0\right).
\]
Then 
\[
\kappa\left(K\right)\ge\dim\left(H\right)-\frac{C}{1-\rho}.
\]
Moreover, equality holds whenever the defect trace is exactly $C\rho^{n}$. 
\end{cor}

\begin{cor}[Scalar moment model]
 In the scalar model of Section 4, where 
\[
K\left(m,n\right)=\int^{1}_{0}x^{m+n}d\mu\left(x\right)\cdot I_{H},
\]
one has 
\[
\kappa\left(K\right)=\dim\left(H\right)\cdot\mu\left(\{1\}\right).
\]
Furthermore, 
\[
\text{tr}\left(\left(\Sigma^{n}K^{(1)}\right)\left(\emptyset,\emptyset\right)\right)=\dim\left(H\right)\int^{1}_{0}\left(1-x^{2}\right)x^{2n}d\mu\left(x\right),
\]
and summing these terms reproduces $\dim\left(H\right)\mu\left(\{1\}\right)$. 
\end{cor}

\begin{proof}
We computed in the Example from Section 4 that 
\[
\left(\Sigma^{N}K\right)\left(\emptyset,\emptyset\right)=\int^{1}_{0}x^{2N}d\mu\left(x\right)\cdot I_{H}.
\]
Since $x^{2N}\to\mathbf{1}_{\{1\}}$ pointwise on $\left[0,1\right]$,
the limit is $\mu\left(\{1\}\right)I_{H}$. Thus 
\[
\kappa\left(K\right)=\dim\left(H\right)\mu\left(\{1\}\right).
\]
The defect formula follows from the identity 
\[
K^{(1)}\left(m,n\right)=\int^{1}_{0}x^{m+n}\left(1-x^{2}\right)d\mu\left(x\right)\cdot I_{H}.
\]
Applying $\Sigma^{n}$ multiplies by $x^{2n}$. Summing the resulting
geometric series on $\left[0,1\right)$ gives the stated relation. 
\end{proof}

\subsection*{Monotonicity and functoriality}

Curvature behaves naturally under compression, dilation, and unitary
similarity.
\begin{prop}
Let $K$ and $L$ be moment kernels with first defects satisfying
$K^{(1)},L^{(1)}\ge0$.
\begin{enumerate}
\item If $0\le L\le K$, then $\kappa\left(L\right)\le\kappa\left(K\right)$.
\item If $U:H\to H$ is unitary and $K^{U}\left(\alpha,\beta\right)=U^{*}K\left(\alpha,\beta\right)U$,
then $\kappa\left(K^{U}\right)=\kappa\left(K\right)$.
\item If $V:H\to H'$ is an isometry and $K'\left(\alpha,\beta\right)=VK\left(\alpha,\beta\right)V^{*}$,
then $\kappa\left(K'\right)=\kappa\left(K\right)$.
\end{enumerate}
\end{prop}

\begin{proof}
(1) Since $\Sigma$ is order-preserving, $\Sigma^{N}L\le\Sigma^{N}K$.
Taking $\left(\emptyset,\emptyset\right)$-entries, applying trace,
and passing to the limit gives $\kappa\left(L\right)\le\kappa\left(K\right)$.

(2) Since $\text{tr}\left(U^{*}XU\right)=\text{tr}\left(X\right)$,
\[
\text{tr}\left(\left(\Sigma^{N}K^{U}\right)\left(\emptyset,\emptyset\right)\right)=\text{tr}\left(U^{*}\left(\Sigma^{N}K\right)\left(\emptyset,\emptyset\right)U\right)=\text{tr}\left(\left(\Sigma^{N}K\right)\left(\emptyset,\emptyset\right)\right),
\]
so $\kappa\left(K^{U}\right)=\kappa\left(K\right)$.

(3) For an isometry $V$, 
\[
\left(\Sigma^{N}K'\right)\left(\emptyset,\emptyset\right)=V\left(\Sigma^{N}K\right)\left(\emptyset,\emptyset\right)V^{*}.
\]
Since $V^{*}V=I_{H}$, 
\[
\text{tr}\left(\left(\Sigma^{N}K'\right)\left(\emptyset,\emptyset\right)\right)=\text{tr}\left(\left(\Sigma^{N}K\right)\left(\emptyset,\emptyset\right)\right).
\]
Passing to the limit gives $\kappa\left(K'\right)=\kappa\left(K\right)$. 
\end{proof}

\subsection*{Interpretation}

The curvature $\kappa\left(K\right)$ is a numerical invariant measuring
the ``persistent energy'' of a random operator tuple. It satisfies: 
\begin{itemize}
\item $\kappa\left(K\right)=0$ if and only if the tuple is fully dissipative
(pure), 
\item $\kappa\left(K\right)=\dim\left(H\right)$ if and only if it is unitary
in mean-square, 
\item intermediate values measure the ``dimension'' of the shift-invariant
component. 
\end{itemize}
The defect-summation formula of \prettyref{thm:6-3} shows that curvature
complements the cumulative mean-square defect, giving a precise balance
law: 
\[
\text{initial mass}=\text{curvature}+\text{total defect}.
\]

\bibliographystyle{amsalpha}
\bibliography{ref}

\end{document}